%% file: main.tex
\documentclass[11pt]{article}

\usepackage[latin1]{inputenc}
\usepackage[T1]{fontenc}
\usepackage{amssymb,graphicx}

\graphicspath{{./images/}{./}}

\newtheorem{Theorem}{Theorem}
\newtheorem{propo}{Proposition}
\newtheorem{coro}{Corollary}
\newtheorem{Lemma}{Lemma}
\newtheorem{Definition}{Definition}

\newenvironment{proof} {{\bf Proof.}}{\hfill $\Box$}

\newcommand{\Sn} {{\mathcal S}_n}

\begin{document}

\begin{center}
{\bf \Large Indecomposable Permutations, Hypermaps and Labeled Dyck Paths}\\[12pt]

{\large Robert Cori}\\[14pt]
{\em Labri, Universit\'e Bordeaux 1\\
351 cours de la Lib\'eration F33400 Talence (France)}\\
{\tt robert.cori@labri.fr}
\end{center}


\vskip 3cm
\begin{abstract}
Hypermaps were introduced as an algebraic tool for  the representation
 of  embeddings of  graphs on an orientable surface. Recently a  bijection
was given between hypermaps and indecomposable permutations; 
 this sheds new light on
the subject by connecting  a hypermap  to a simpler object.\ In this paper, a
bijection between indecomposable  permutations and labelled Dyck paths is
proposed, from which a few enumerative results concerning
hypermaps and maps follow. We obtain for instance an inductive
 formula for the number of
hypermaps with $n$ darts, $p$ vertices and $q$ hyper-edges; the latter is
also  the number of indecomposable permutations of
${\mathcal S}_n$ with $p$ cycles and $q$ left-to-right maxima.
The distribution of these parameters among all permutations is also considered.
\end{abstract}

\section*{Introduction}

Permutations and maps on surfaces have and old common history.
Heffter \cite{heffter}
was probably  the first who mentioned the fact  that any embedding of a
graph on an orientable  surface could be represented by a pair consisting of a
permutation and a fixed point free involution;  J. Edmonds \cite{edmonds}
and J. Youngs \cite{youngs}
 gave in the early
60's a more precise presentation of this  idea by showing how
to   compute the faces  of an  embedding
using the  cycles of the product of the  permutation and
 the fixed point free involution, giving a purely combinatorial
definition of the genus.
 A. Jacques \cite{jacques} proved that this
 could be generalized to
any  pair of permutations (called hypermap in \cite{coriT}), 
hence relaxing the condition  that the
second one should be a fixed point free involution.
He defined the genus of a pair of permutations
by a formula  involving the number of
their  cycles and of that of their product.

W. T. Tutte \cite{tutte2} generalized these constructions by introducing
a combinatorial object consisting of
three fixed point free involutions
 in order to represent embeddings in
 a nonorientable surface.

The combinatorial representation  allows one to obtain
 results on automorphisms
of maps and hypermaps, for instance  A. Mach\'{\i} \cite{machi} obtained a
combinatorial version of the Riemann--Hurwitz formula for hypermaps.
 A coding theory of
rooted maps  by words \cite{coriT}  had also some extent for
explaining the very elegant formulas found by W. T. Tutte \cite{tutte1} for the
enumeration of maps.\ In the same years Jones and Singerman
 \cite{jonesSingerman}
settle some important algebraic properties of maps.
Recently G. Gonthier
see \cite{gonthier}, used hypermaps in giving a formal proof
of the 4 colour theorem.
A survey of the combinatorial and algebraic properties of maps and
hypermaps is given in   \cite{coriMachi}.

\bigskip

In 2004 P. Ossona de Mendez and P. Rosenstiehl proved an important
combinatorial result: they constructed a bijection
between (rooted) hypermaps and 
indecomposable permutations (also called connected or
irreducible). Indecomposable permutations are a central object
in combinatorics (see for instance \cite{stanley} Problem 5.13), 
they were  considered in different contexts, and
probably for the first time by A. Lentin \cite{LentinT} while  solving
 equations in the free monoid.
They were also considered for instance  as a basis of a Hopf
algebra defined by Malvenuto and Reutenauer (see \cite{aguiarSottile}, or
 \cite{duchampHivert}), and  in the enumeration  of a certain kind of Feynman 
diagrams \cite{cvitanovic}.

\bigskip

In this paper we  present the Ossona-Rosenstiehl result in  simpler terms and
focus on the main property of the bijection: the number of cycles
and of left-to-right maxima of the
indecomposable permutation are equal to
the
number of vertices and hyper-edges of the rooted hypermap associated with 
it.

This property have some nice consequences for the enumeration: 
it allows one to give a formula for the number of rooted hypermaps on $n$
darts, or with $n$ darts and $p$ vertices. The property shows that the
number of indecomposable permutations of $\Sn$ with $p$ cycles and $q$
left-to-right maxima is symmetric in $p, q$. By a straightforward
argument this result can be generalized to all permutations, answering
positively to a conjecture of Guo-Niu Han and D. Foata. We  introduce a simple
bijection between some labelled  Dyck paths and permutations
which allows us to obtain  a formula
for the polynomials enumerating indecomposable permutations by the number
of cycles and left-to-right maxima, (hence of hypermaps by vertices and
hyper-edges).

The  paper is organized as follows: in Section 1, we give a few 
elementary results on indecomposable permutations focusing mainly on
the parameters  left-to-right maxima and cycles. Section 2 is devoted
to hypermaps and  the bijection of P. Ossona de Mendez and P. Rosenstiehl.\ All the details of
the  proof of correctness are in \cite{ossonaRosenstiehl1}; we give
here the key points and some examples in order to facilitate
the reading of their paper. The main result of the present paper and consequences of this bijection are 
given at the end at this section

In Section 3 we recall some notions about Dyck paths  and their 
labelling. We  describe a bijection between them and permutations and
show the main properties of this bijection. In Section 4  we
introduce a family of polynomials enumerating permutations and
show 
 a formula for the
generating function of the permutations with respect to the number of
left-to-right maxima and cycles. In the  last section
we restrict the hypermaps to be maps and permutations to be fixed point free
involutions, obtaining in a simpler way some old
enumeration formulas for them (see \cite{walshLehman}, \cite{arquesBeraud}).

\include{contenu}

\bibliography{permutations}
\bibliographystyle{plain}

\end{document}

%% file: contenu.tex
\section{Indecomposable permutations}

In this section we give some notation
and recall some basic results on permutations.\
Then we shall focus on indecomposable permutations,
often called also connected permutations.

\subsection{Definition and counting formulas}
Permutations are the central object of this paper.\ We shall express them in two ways either as {\em
  sequences}, or as {\em sets of cycles}.

The set of all permutations (i. e. the symmetric group) on
$\{1,2, \ldots, n\}$ will be denoted by $\Sn$.
The notation of a permutation as a sequence is :
$$\alpha \ \ = \ \  a_1, a_2, \ldots,  a_n$$
In this setting, $a_i$ is the image of $i$ by $\alpha$, also denoted
 $\alpha(i)$.

The product of two
 permutations $\alpha, \beta$ is  denoted $\alpha\beta$ for which
 we use the following convention
$\alpha\beta(i) = \alpha(\beta(i))$.

\begin{Definition}
A permutation $\theta = a_1, a_2, \ldots, a_n$ is {\em decomposable},
if there exists $p < n$ such that for all $i, \ \ 1 \leq i \leq p$~:
$$ 1 \leq  a_i \leq p, $$ it is called
{\em indecomposable} otherwise.
\end{Definition}

Hence a  permutation $\theta = a_1, a_2, \ldots, a_n$ is indecomposable
if for any $p < n$
the left factor
subsequence $a_1, a_2, \ldots , a_p$ contains at least one  $a_j > p$.
Equivalently, $\theta$ is indecomposable if  for $p<n$,
 there is no initial interval
 $[1, \ldots , p]$ fixed by $\theta$, or no union of a subset of
the set of  cycles equal to  $[1,  \ldots , p]$.\ For instance,
3,1,2,5,4 is a decomposable permutation of $S_5$, while 
any permutation $\alpha=a_1,a_2,\ldots, a_n$ for which $a_1=n$ is indecomposable.

Let  $c_n$ be the number of indecomposable permutations
of $\Sn$.\ The following formula is well known, and is obtained by
noting that any decomposable permutation
can be written as the product of an
 indecomposable permutation of length
 $p< n$  and a permutation of length $n-p$ on $\{p+1, \ldots , n\}$~:

$$ c_n = n! - \sum_{p=1}^{n-1}c_p(n-p)!$$

\noindent
{\bf Remark}
Less known is the following formula, which is  also useful to compute
$c_n$~:
$$ c_{n+1} \ \ = \ \ \sum_{p=1}^{n}pc_p(n-p)!$$
this formula is obtained from the following construction of all
indecomposable permutations of ${\mathcal S}_{n+1}$~:

\begin{itemize}
\item
take any permutation $\alpha$ of $\Sn$
\item decompose $\alpha$ as the concatenation
of an indecomposable permutation $\theta$ of length $p>0$ and a permutation
 $\beta$ of length $n-p$
\item insert $n+1$ inside $\theta$ in any position except at the end of it
(there are $p$ such positions).
\end{itemize}

From these formulas we obtain the  first values of the number of
 indecomposable permutations which are~:
$$1, 1, 3, 13, 71, 461, 3447, \ldots $$

\subsection{Left-to-right maxima and cycles}
Let $\alpha = a_1, a_2, \ldots , a_n$ be a permutation,
$a_i$ is a left-to-right maximum if $a_j < a_i$ for all $1 \leq j < i$.

\noindent
For any  $\alpha$,  $a_1$ is a left-to-right maximum,  and  $a_k= n$
is also a left-to -right maximum, hence  the number
of left-to-right maxima of a  permutation
$\alpha$ is equal to 1 if and only if
 $a_1 = n$.

\subsubsection{Bijection}
The following algorithm describes  a bijection from the set of
 permutations  having  $k$ cycles
to the set of permutations having  $k$ left-to-right maxima. It is
often  called the {\em First fundamental  transform} and extensively used
for the determination of permutation statistics (see \cite{foataSchutz}).

To obtain the transform $\beta$ from $\alpha$, write the cycles 
 $\Gamma_1, \Gamma_2 , \ldots , \Gamma_k$ of the
 permutation
 $\alpha$,\ 
such that the first element of each cycle  $\Gamma_i$
 is the maximum among  the elements of
 $\Gamma_i$.\ Then
 reorder the  $\Gamma_i$ in such a way that the first elements of the cycles
 appear in increasing order, and finally,
delete  the parenthesis around the cycles obtaining $\beta$
 as a sequence.

\medskip

For instance, let $\alpha = 4, 7, 2, 1, 3, 6, 5, 9, 8$,
then we write~:

$$ \alpha = (1, 4) (2, 7, 5, 3) (6) (8, 9)$$

\noindent
putting the maximum at the begin of each cycle and
reordering the cycles gives~:

$$ \alpha = (4, 1) (6) (7, 5, 3, 2)  (9, 8)$$

\noindent
hence
$$ \beta = 4, 1, 6, 7, 5, 3, 2, 9, 8$$

Note that one gets $\alpha$ from its transform $\beta$ by
opening
parenthesis  before each left-to-right maxima, and closing
them before opening a new one and at the end of the sequence.

\begin{propo} The permutation  $\alpha$  is indecomposable if and only if
its first fundamental  transform is indecomposable.
\end{propo}

\begin{proof}
If $\alpha$ is indecomposable, then for any $p < n$
the subset $\{1,2,\ldots, p\} $ is not the   union  of  cycles of $\alpha$.\
Hence for any left factor  $b_1, b_2, \ldots, b_p$ of $\beta$,
 where $b_p$ is the last
element of a cycle of $\alpha$ their exists a $b_j$ such that $b_j > p$.
But this is also true if $b_p$ is not the last element of the cycle
of $\alpha$
since the maximal element of each cycle of $\alpha$
 is put at the beginning of its
cycle in order to obtain $\beta$.
The proof of the converse is obtained  by similar  arguments.
\end{proof}
\medskip

\noindent
{\bf Corollary.} \ {\it The  number of  indecomposable
 permutations with  $k$ cycles
is equal to the number of indecomposable  permutations with
 $k$ left-to-right maxima}.

\medskip

\noindent
{\bf Remark.}\ It is customary for permutations to consider
 in an obvious manner
 left-to-right minima, right-to-left maxima and right-to-left minima.
Consider for a  permutation $a_1, a_2, \ldots , a_n$, the reverse
$a_n, a_{n-1}, \ldots , a_1$   and the complement $n+1-a_1, n+1-a_2, \ldots ,
n+1- a_n$; these operations show that the statistics for these four parameters
 are equal. However
this is not true  for indecomposable permutations.\ For instance, the numbers
of indecomposable permutations of ${\mathcal S}_4$
with $k=1,2,3, 4$ left-to-right minima are
0, 7, 5, 1, respectively, and those with the same number of left-to-right maxima
are $6, 6, 1, 0$.
Nevertheless the operation $\alpha \rightarrow \alpha^{-1}$
transforms an indecomposable permutation into an indecomposable permutation,
 showing that the number of indecomposable permutations of $\Sn$  with
$k$ left-to-right maxima (resp. minima) is equal to the number of
 indecomposable permutations of $\Sn$  with $k$ right-to-left minima
 (resp. maxima).

\subsubsection{Enumeration}

It is well known that the  number of permutations  $s_{n,k}$
of  ${\mathcal S}_n$ having
 $k$ cycles is equal to the
coefficient of  $x^k$ in the polynomial~:

 $$ A_n(x) \ \ = x(x+1)(x+2)  \cdots (x+n-1)$$

\noindent
These numbers are the unsigned Stirling numbers of the first kind.

\begin{propo}
The number $c_{n,k}$ of indecomposable permutations
of $S_n, \ \ n>1$ with $k$ cycles
(or with $k$  left-to-right maxima)  is
given by each one of the following formulas~:
$$ 
c_{n,k}= s_{n,k} -\sum_{p=1}^{n-1} \sum_{i=1}^{min(k,p)}  c_{p,i}s_{n-p,k-i}
\ \ \ \ \ \ \ \ \ 
c_{n,k}= \sum_{p=1}^{n-1} \sum_{i=1}^{min(k,p)} p c_{p,i}s_{n-p-1,k-i}
$$

 where  $s_{m,j}$ is the number  of  permutations of  ${\mathcal S}_m$
with $j$ cycles .
\end{propo}

 \begin{proof} \
The first formula follows from the observation that a decomposable
permutation of $\Sn$ with $k$ cycles is the concatenation 
of an indecomposable permutation of ${\mathcal S}_p$ 
with $i$ cycles and a permutation 
of ${\mathcal S}_{n-p}$ with $k-i$ cycles. 

For the second one, observe that the deletion of $n$ from its cycle 
in  an indecomposable  permutation $\alpha$  of $\Sn$ 
 with $k$  cycles gives a (possibly decomposable) permutation with $k$ cycles~;
 since if $n$ was alone in
its cycle then $\alpha$ would have been decomposable.
Conversely, let $\beta$ be any permutation with $k$ cycles written as
the concatenation of an indecomposable permutation $\theta$ on  $\{1,2,
\ldots , p\}$  (with $p \leq n$) and a permutation $\beta'$ on $\{p+1
\ldots , n-1\}$ having respectively $i$ and $k-i$ cycles. When inserting
$n$ in any cycle of $\theta$ one gets an indecomposable permutation
with $k$ cycles. The formula follows from the fact that there are exactly $p$
places where $n$ can be inserted, since inserting $n$ in a cycle of
$\beta'$ gives a decomposable permutation.
\end{proof}

Let $C_n(x) = \sum_{k=1}^{n-1} c_{n,k}x^k$; the equalities become
$$
C_n(x) = A_n(x) - \sum_{p=1}^{n-1} A_{n-p}(x)C_p(x) \ \ \ \ \ \ \ \ 
C_n(x) = \sum_{p=1}^{n-1} p A_{n-1-p}(x)C_p(x)$$
\medskip

{\bf Indecomposable Stirling numbers.}
The first values of the number of indecomposable permutations of $\Sn$
with $k$ cycles, for $2\leq n \leq 7$, are given in the table below; these
numbers might be called indecomposable Stirling numbers of the first kind
since they count
 indecomposable
permutations by their number of cycles.

\begin{center}
\begin{tabular}{cccccc}
1\\
2& 1\\
6& 6 & 1\\
24& 34 & 12 & 1\\
120 & 210 & 110 &20 & 1\\
720 & 1452 & 974 & 270 & 30 & 1
\end{tabular}
\end{center}

\section{Hypermaps}
In this section we recall some elementary facts about hypermaps,
state the main result of P. Ossona de Mendez and P. Rosesntiehl and
give a simplified proof of it.

\subsection{Definition}
Let $B$ be a finite set the  elements of it being  called {\em darts}.
In the sequel we will take $B= \{1, 2, \ldots , n\}$.

\begin{Definition}
A {\em hypermap} is given by a   pair of permutations
$(\sigma, \alpha)$, acting on  $B$
such that the group they generate is  transitive on $B$.
\end{Definition}

\medskip
The transitivity  condition can be translated in simple combinatorial terms,
remarking  that it is equivalent to the connectivity
of  the graph $G_{\sigma, \alpha}$  with vertex set $B$ and edge set~:
$$E = \bigcup_{b\in B} \{b,\alpha(b)\}  \bigcup_{b\in B} \{b,\sigma(b)\}$$

The cycles of $\sigma$ are the {\em vertices} of the hypermap
while the cycles of $\alpha$ are the {\em hyper-edges}.

An example of hypermap with 3 vertices and three hyper-edges
is given by~:

$$\sigma = (1, 2, 3) (4, 5, 6) (7, 8, 9) \ \ \ \ \alpha = (1, 6, 7)
(2, 5, 8) (3, 4, 9)$$

Hypermaps have been introduced, as a generalisation of combinatorial maps,
 for the representation of embeddings of
hypergraphs in surfaces, showing that  the cycles of
$\alpha^{-1}\sigma$
can be considered as representing the
faces, and defining a genus in a formula like Euler's one for maps.
In this paper we will not consider hypermaps as a topological embedding but as
 the very simple object consisting of a
 pair of permutations generating a
transitive subgroup of $\Sn$.

\subsection{Labeled, unlabeled and rooted hypermaps}
In enumerative combinatorics it is customary to consider labeled
objects and unlabeled ones.
Since in the above definition of
hypermaps we consider the elements of $B$ as distinguishable numbers,
they should be called  {\em labeled hypermaps}.

\smallskip

As an example, the number
of labeled hypermaps with 3 darts is 26, since among the 36 pairs of
permutations on $\{1, 2, 3\}$ there are 10 of them which do not generate a
transitive group.\ These are given by (where $\varepsilon$ is the identity
and $\tau_{i,j}$ the transposition exchanging $i$ and $j$)~:
\begin{itemize}
\item $\sigma = \varepsilon$ and $\alpha =\tau_{i,j}$ $i \neq j
  \in\{1,2,3\}$
 or $\alpha =
  \epsilon$  (4   pairs)
\item $\sigma = \tau_{i,j}$  and $\alpha  = \varepsilon$ or $\alpha = \sigma$
  (2   pairs for each of the 3 transpositions).
\end{itemize}
\
Two hypermaps $(\sigma, \alpha)$ and  $(\sigma', \alpha')$
are {\em isomorphic} if there exists a permutation $\phi$ such that~:
$$ \phi^{-1} \alpha \phi \ = \ \alpha',
 \ \ \ \ \ \ \  \phi^{-1} \sigma
\phi  \ = \  \sigma'.$$
The set of {\em unlabeled hypermaps} is the quotient of the set
of labeled ones by the isomorphism relation.
For instance the number of unlabeled hypermaps with 3 darts is 7.\ Representatives of the 7 isomorphism classes are given below:
$$
\begin{array}{|c|c|c|c|c|c|c|c|}
\hline
& H_1& H_2 & H_3& H_4 & H_5 & H_6 & H_7\\
\hline
\sigma&(1,2,3)& (1,2,3)& (1,2,3)& (1,2,3)& (1,2)(3)& (1)(2,3)& (1)(2)(3)\\
\hline
\alpha&(1,2,3)& (1,3,2)& (1,2)(3)& (1)(2)(3)& (1,2,3)& (1,3) (2)& (1,2,3)\\
\hline
\end{array}
$$
In general, the enumeration of unlabeled objects is difficult and the
formulas one obtains are  complicated.\ Thus,
intermediate objects are introduced: the  {\em rooted} ones, this is
done by selecting one element,  {\em the root},  in the object, and considering
isomorphisms which fix the root.
For hypermaps we select $n$ as the root, two labeled hypermaps $(\sigma, \alpha)$ and
$(\sigma', \alpha')$ being {\em isomorphic} as {\em rooted hypermaps}
if there exists $\phi$ such that~:
$$ \phi^{-1} \alpha \phi \  = \ \alpha', \ \ \ \  \phi^{-1}
  \sigma \phi \  = \ \sigma', \ \ \ \ \  \phi(n) = n$$

Such a $\phi$ will be called a {\em rooted isomorphism}.
There are 13 different rooted hypermaps with 3 darts, to the 7 above
 we have to add these below, which are isomorphic to one of the
 previous  ones but for
 which the isomorphism $\phi$  is not a rooted isomorphism.
\[
\begin{array}{|c|c|c|c|c|c|c|}
\hline
& H_8& H_9&  H_{10} & H_{11} & H_{12} & H_{13}\\
\hline
\sigma&(1,2,3) & (1,2,3)& (1)(2,3)& (1)(2,3)& (1,2)(3)& (1)(2,3)\\
\hline

\alpha&(1,3)(2)& (1)(2,3)& (1,3,2)& (1,2,3)& (1,3)(2)& (1,2) (3)\\
\hline
\end{array}
\]
In the sequel we will denote by $h_n$ the number of rooted hypermaps
with $n$ darts.
\begin{propo}
The number of labeled hypermaps with $n$ darts is equal to
$$(n-1)!h_n$$
\end{propo}

\begin{proof}
Since there are $(n-1)!$ permutations $\phi$ such that $\phi(n) = n$,
 we only have
to prove that for a hypermap $(\sigma, \alpha)$ and a $\phi$ such
that $\phi(n) = n$ if
$$ \phi^{-1} \alpha \phi \  = \  \alpha, \ \ \ \  \phi^{-1} \sigma
\phi \  = \  \sigma$$
then $\phi$ is the identity. But this follows from the fact
 that for such an isomorphism $\phi$, $\phi(a) = a$ implies
$\phi(\sigma(a) ) = \sigma(a), \ \ \phi(\alpha(a)) =\alpha(a)$ and
from the transitivity of the group generated by  $\sigma$ and  $\alpha$.
\end{proof}

\subsection{Bijection}

 The following algorithm is a slightly modified version  of
 that of  P.  Ossona de Mendez, and P. Rosenstiehl (\cite{ossonaRosenstiehl1})~:

\subsubsection{Algorithm OMR}

 Let
$\theta= a_1,  a_1, a_2, \ldots , a_{n+1}$ be  an indecomposable permutation.\
A pair of permutations $(\sigma, \alpha)$
 is associated with $\theta$ through the following algorithm~:

\begin{itemize}

\item  Determine the  left-to-right maxima  of
 $\theta$, that is the
indices   $i_1, i_2, \ldots , i_k$ satisfying $j < i_p \Rightarrow a_j < a_{i_p}$.
Note that
$i_1 = 1\ \ $ and $ \ a_{i_k}= n+1$

\item  Let $\sigma_1$ be the permutation split into cycles as:

$$\sigma_1 \ \ = \ \  (1,2, \ldots,  i_2-1)(i_2, i_2+ 1, \ldots, i_3-1) \ldots
(i_k,  \ldots ,n+1)$$

\item  The  permutations $\alpha$ and  $\sigma$
are  obtained   from $\theta$ and $\sigma_1$, respectively, by deleting
$n+1$ from their cycles (observe that these cycles are  of length
not less than 2).
\end{itemize}

We will denote by $\Psi(\theta) $ the pair of permutations
 $(\sigma, \alpha)$
 obtained from $\theta$ by means of the algorithm OMR.

\medskip

\noindent
{\bf Example}
Consider the indecomposable permutation
$$ \theta = 6, 5, 7,  4, 2, 10, 3, 8, 9,  1$$
then the indexes of the left-to-right maxima
are $1, 3, 6$ giving
$$ \sigma_1 = (1, 2) (3, 4, 5) (6, 7, 8, 9, 10)$$
Since $\theta = (1, 6, 10) (2, 5) (3,7) (4) (8) (9)$
we have

$$\sigma = (1, 2)  (3, 4, 5) (6, 7, 8, 9) \ \ \ \ \alpha = (1, 6)
(2, 5)  (3, 7) (4) (8)  (9)$$

\medskip

\begin{Theorem}\label{thBij2}
 The above algorithm yields a  bijection $\Psi$ between the set of  indecomposable
 permutations
on ${\mathcal S}_{n+1}$ and the set of rooted hypermaps with darts
 $1,2, \ldots , n$.
Moreover for   $(\sigma, \alpha) = \Psi(\theta)$,
 $\alpha$ and  $\theta$ have the same  number of cycles  and
the number of cycles of $\sigma$ is equal to the number of
left-to-right maxima of $\theta$.
\end{Theorem}

The key point in the proof of this Theorem is the following characterization
of the smallest elements of the cycles of the permutation $\sigma$ given by
the algorithm:

\begin{Lemma}\label{lem1}
A hypermap $(\sigma, \alpha)$ is such that there exists an indecomposable 
permutation $\theta$ satisfying 
$$ \Psi(\theta) = (\sigma, \alpha)$$
if and only if 
\begin{itemize}
\item The permutation $\sigma$ has cycles consisting of consecutive
  integers in increasing order~:
$$\sigma \ \ = \ \ (1, 2, \ldots , i_2 -1) (i_2, i_2+1, \ldots , i_3 -1)
\ldots (i_k, i_k+1, \ldots , n) $$
\item $\alpha$ is such that the right-to-left minima of $\alpha^{-1}$
  are $i_1, i_2, \ldots , i_{k-1}$ and a (possibly empty) subset of the
interval $[i_k, i_k+1, \ldots,  n]$.
\end{itemize}

\end{Lemma}
\begin{proof}
Suppose that $(\sigma, \alpha)$ satisfy the conditions above, consider the notation of $\alpha \ \ = \ \ a_1, a_2, \ldots , a_n$ as a sequence and 
let 

$$ \theta \ \ = \ \  a_1, a_2, \ldots a_{{i_k}-1}, n+1, a_{{i_k}+1}, \ldots , a_n,  a_{{i_k}}$$

Then $\theta(i_k) = n+1$ and $\theta(n+1) \ = \ \alpha(i_k)$,
 hence the indexes of the
left-to-right maxima of $\theta$ are $i_1, i_2, \ldots , i_k$ 
giving $\Psi(\theta) =
(\sigma, \alpha)$.

\smallskip

Conversely let $\theta$ be an indecomposable permutation and let 
$= \Psi(\theta)=(\sigma, \alpha) $, then~:
\begin{enumerate}
\item By the definition of $\Psi$ the cycle of $\sigma$ containing $n$
  is $(i_k, \ldots,  n)$.
\item  For any permutation $\theta$,
the indexes $i_1, i_2, \ldots, i_k$ of the left-to-right maxima are exactly the
right-to-left minima of $\theta^{-1}$.  Deleting $n+1$ from its
cycle in $\theta$ in order to obtain $\alpha$ has the following effect
on the sequence $b_1, b_2, \ldots , b_{n+1}$ representing
$\theta^{-1}$: $b_i=n+1$ is replaced by $b_{n+1}$. Clearly,
$i_1, i_2, \ldots i_{k-1}$,  are still right-to-left minima  in
the sequence obtained  by this transformation
since they are smaller than $b_{n+1}$.
\end{enumerate} \end{proof}

\medskip

\noindent
\subsubsection{Proof of  Theorem \ref{thBij2}}.
\begin{enumerate}
\item
The   pair of permutations $(\sigma, \alpha)$
defines a hypermap.

Consider a cycle $(i_p, i_p + 1, \ldots , i_{p+1} -1)$ of $\sigma$ with
$p<k$, then $\alpha(i_p)= \theta(i_p)$ is a left-to-right maximum of $\theta$
and the next one is $\alpha(i_{p+1})$, hence $\theta(i) < \theta(i_p)$ 
for $i_p< i < i_{p+1}$ and for $i < i_p$. Since $\theta$ is indecomposable this 
implies $\theta(i_p) = \alpha(i_p) \geq i_{p+1}$. We have thus observed that for
any cycle $\Gamma_p$ of $\sigma$, which does not contain $n$, there is an element (namely $i_p$)
of it such that $\alpha(i_p)$ is in another cycle of $\sigma$ which smallest element 
is greater than  the smallest element of $\Gamma_p$; this observation clearly implies
the transitivity of the group generated by $(\sigma, \alpha)$.

\item Let  $\theta$ and $\theta'$ two different indecomposable permutations
then $\Psi(\theta)$ and $\Psi(\theta')$ are non isomorphic as rooted
hypermaps.

\begin{itemize}

\item Suppose that there exists a rooted isomorphism $\phi$ between
$\Psi(\theta) = (\sigma, \alpha) $ and $\Psi(\theta') = (\sigma',
\alpha')$, then the cycles of $\sigma$ and $\sigma'$ containing $n$
have the same length, since $\phi(n) = n$ and $\sigma' = \phi^{-1}\sigma \phi$ implies~:
$$ \sigma^i(n) \ \  =\ \  n \ \ \ \Leftrightarrow \ \ \  \sigma'^i(n) \ \ = \ \ n $$
\item Hence $i_k$ and $i'_{k'}$ the smallest elements of the cycles
of $\sigma$ and $\sigma'$ containing $n$ are equal, and
$\phi(i) = i$ for all $i_k \leq i \leq n$.
\item By  Lemma \ref{lem1}, $i_{k-1}$ is equal to $\alpha^{-1}(j)$ for
  the maximal $j$ in $[i_k, \ldots , n]$ such that
$\alpha^{-1}(j) \notin [i_k, \ldots , n]$. But since $\phi(i) = i$ for
$i \in [i_k, \ldots, n]$ we have:
$$ \alpha^{-1}(\ell) \notin [i_k, \ldots , n] \ \ \Leftrightarrow \ \
\alpha'^{-1}(\ell) \notin [i_k, \ldots , n]$$
Hence $\alpha'^{-1}(j) \notin [i_k, \ldots , n]$, and 
$i'_{k-1} = \alpha'^{-1}(j)$. Moreover the cycles
of $\sigma$ and $\sigma'$ containing respectively
 $i_{k-1} $ and $i'_{k'-1}$ have the same length, giving
$i_{k-1} = i'_{k'-1}$ and $\phi(i) = i$ for all $i_{k-1} \leq i \leq n$
\item  By repeating the above argument for all the $i_p$ we
  conclude that $\phi$ is the identity.
\end{itemize}

\item For any hypermap $(\sigma, \alpha)$ there exists an
  indecomposable permutation $\theta$ such that  $(\sigma, \alpha)$
and  $\Psi(\theta)$ are  isomorphic as rooted
hypermaps.

It suffices to show that there exists an
isomorphism $\phi$ such that the hypermap  $(\sigma', \alpha') = (\phi^{-1}\sigma \phi, \phi^{-1}\alpha \phi)$  satisfies the conditions of Lemma \ref{lem1}.
We have to find among all the conjugates of $\sigma$ one in which the cycles
consist of consecutive integers and such that the smallest elements
 in each cycle 
are the right-to-left minima of the conjugate of $\alpha^{-1}$. 
For that we  write down the cycles $\Gamma_1, \Gamma_2, \ldots , \Gamma_k$
of $\sigma$ in a specific order, then we will 
write $\sigma'$ (having the same numbers of cycles of each length as $\sigma$
and with cycles consisting of consecutive numbers in increasing order) above
$\sigma$
in such a way that cycles of the same length
 correspond. The automorphism $\phi$ is then
 obtained by the classical construction
for conjugates of a permutation (see for instance \cite{Rotman} chapter 3).

Since $\phi(n) = n$, $\Gamma_k=   (z_{j_1}, z_{j_2}, \ldots, z_{j_k} )$
 should be the cycle of $\sigma$ containing $n$ and 
it has to be written such that $z_{j_k}= n $.
In order to find which cycle is $\Gamma_{k-1}$ we use Lemma \ref{lem1}:
 the first element
of this cycle should correspond to a right-to-left minima of $\alpha'^{-1}$,
hence this element is the first among 
 $\alpha^{-1}(z_{j_k}), \alpha^{-1}(z_{j_{k-1}}) \ldots
\alpha^{-1}(z_{j_1})$ which is not in $\Gamma_k$, such element exists 
by the transitivity of the group generated by $\{\sigma, \alpha\}$.
 We  continue by computing
 the image under $\alpha^{-1}$ of the elements already
written down, taken from right to left, when an element
 $\alpha^{-1}(z_i) = u_1$, not written down is obtained, the
whole cycle of $\sigma$ containing $u_1$ is written with $u_1$ at the
beginning of it. The algorithm terminates when all $\{1,2 , \ldots , n\}$ are
obtained, and this termination is also a consequence of  the
 transitivity of the   group generated by $\{\sigma, \alpha\}$.

To end we write $\sigma'$ above $\sigma$ in such a way that
the elements $1, 2, \ldots , n$ appear in that order with the lengths
of cycles corresponding to those of $\Gamma_1, \Gamma_2, \ldots , \Gamma_k$,
 and the isomorphism
$\phi$ is determined. Then $\alpha'$ is obtained by  
$\alpha' = \phi^{-1}\alpha \phi$ \hfil $\Box$
\end{enumerate}

\subsubsection{Example}
We give an example of a hypermap $H= (\sigma, \alpha)$
and the  computation of the hypermap $H' = (\sigma', \alpha')$
such that $H$ and $H'$ are isomorphic as rooted hypermaps and
there exists $\theta$ satisfying $\Psi(\theta') = (\sigma', \alpha')$.
We take for this example, 
the hypermap obtained by exchanging vertices and hyper-edges
in the hypermap considered above:

$$\sigma  = (1, 6) (2, 5)  (3, 7) (4) (8)  (9) \ \ \
\alpha = (1, 2)  (3, 4, 5) (6, 7, 8, 9). \ \ \ \ $$
We begin the list of cycles by $\Gamma_6 = (9)$, then 
since $\alpha^{-1} (9) = 8$, after two steps  the list consist of
$$ (8) (9)$$
Now since $\alpha^{-1} (8) = 7$, the list grows,
$$(7, 3)(8) (9)$$
then $\alpha^{-1} (3) = 5$ gives:
$$  (5, 2) (7, 3)(8) (9)$$
Then, since  $\alpha^{-1} (7) = 6$:
$$ (6,1) (5, 2) (7, 3) (8) (9)$$
We end by  $\alpha^{-1} (2) = 1, \alpha^{-1} (5) = 4$ and obtain finally;
$$(4)  (6,1) (5, 2) (7, 3) (8) (9).$$
Aligning with
$$(1) (2,3) (4,5) (6,7)(8)(9)$$
we obtain $\phi\ = \ 4,6,1,5,2,7,3,8,9$ then we have~:
$$\sigma' = \phi^{-1} \sigma \phi = 
 (1)(2,3) (4,5) (6,7)(8)(9) \ \ \ \ \alpha'= \phi^{-1} \alpha \phi =  (3,5) (7,1,4) (2,6,8,9)$$
To obtain $\theta'$ we remark that the last cycle of $\sigma$ is of
length 1, hence the position of $10$ in the
sequence representing $\theta$  should be one place before the
end, giving from $\alpha' = 4, 6, 5, 7, 3, 8, 1, 9 , 2$:
$$ \theta' =  4, 6, 5, 7, 3, 8,
1, 9, 10, 2$$

\subsection{Main results}

\begin{coro}
The number of rooted hypermaps with $n$ darts is equal to $c_{n+1}$,
the number of those with $n$ darts and $k$ vertices is $c_{n+1,k}$.
\end{coro}
We also obtain another proof of a result of J. D. Dixon \cite{dixon2}.
\begin{coro}
The probability $t_n$ that a pair of permutations randomly chosen among the
the permutations in $\Sn$ generates a
transitive group is~:

$$p_n = 1 -\frac{1}{n} -\frac{1}{n^2} -\frac{4}{n^3} -\frac{23}{n^4}
-\frac{171}{n^5} -\frac{11542}{n^6}  -\frac{16241}{n^7}
-\frac{194973}{n^8} +O(\frac{1}{n^9})$$

\end{coro}

\begin{proof}
We have seen that the number of labeled hypermaps with $n$ darts is $(n-1)!c_{n+1}$;
hence the probability $t_n$ is
$$\frac{(n-1)!c_{n+1}}{n!n!} \ \ = \ \ \frac{c_{n+1}}{nn!}$$
In \cite{comtetN} L. Comtet proves that the number of indecomposable
permutations $c_n$ of $\Sn$  satisfies
$$ \frac{c_n}{n!} \ \ = \ \  1 -\frac{2}{n}
 -\frac{1}{(n)_2}  -\frac{1}{(n)_3}  -\frac{19}{(n)_4}  -\frac{110}{(n)_5}
-\frac{745}{(n)_6}  -\frac{5752}{(n)_7}  -\frac{49775}{(n)_8} + O(\frac{1}{476994})
$$
where $(n)_k = n(n-1) \ldots (n-k+1)$.\ Replacing $n$ by $n+1$ gives the result.

\end{proof}

\noindent
The following theorem answers positively a conjecture of Guo-Niu Han and
D.  Foata. \footnote{Personal communication.}
\begin{Theorem}
The number of permutations of $\Sn$ with $p$ cycles and $q$ left-to-right
 maxima is equal to the number of permutations of $\Sn$ with $q$
cycles and $p$ left-to-right maxima.
\end{Theorem}

\begin{proof} We  define a bijection $\Phi$ between these two subsets. 
Let $\theta$ be an indecomposable
  permutation with $p$ cycles and $q$ left-to-right maxima, we define
$\Phi(\theta)$ as $\Psi^{-1}(\alpha', \sigma')$ where  
$\Psi(\theta)= (\sigma, \alpha)$, and 
$(\alpha', \sigma')$ is the unique hypermap isomorphic to 
 $(\alpha, \sigma)$ as a rooted
hypermap and  such that there exists an indecomposable permutation
$\theta'$ satisfying  $\Psi(\theta') = (\alpha', \sigma')$.

Clearly,  $\theta'$ has $q$ cycles and $p$
left-to-right maxima as $\theta$.
Then  $\Phi$ is a bijection among
indecomposable permutations having the desired property.

Now a decomposable  permutation $\beta$ can be written as the concatenation
of $k$ indecomposable ones~:
$$ \beta \ \ = \ \ \theta_1 \theta_2 \ldots \theta_k$$
Define $\Phi(\beta)$ by:

 $$\Phi(\beta) = \Phi(\theta_1) \Phi(\theta_2) \ldots \Phi(\theta_k)$$
with an obvious convention on the numbering of the elements
on which the $\theta_i$ act.
Clearly $\Phi(\beta)$ has also has as many cycles as $\beta$ has left-to-right
maxima and as many left-to-right maxima as $\beta$ has cycles,
completing the proof.
\end{proof}

\section{A bijection with labeled Dyck paths}

The  bijection  described below will allow us to obtain a formula 
for the number of
indecomposable permutations with a given number of cycles and of
left-to-right maxima.

\subsection{Dyck paths and labeled Dyck paths}

A Dyck path can be defined  as a word $w$  on the alphabet
$\{a,b\}$ where the number of  occurrences of the letter $a$
(denoted by $|w|_a$),
is equal to the number of occurrences of
  the letter $b$, and such that any left factor contains
no more occurrences of the letter $b$ than those of $a$.
We write:
$$ |w|_a = |w|_b \ \ \makebox{\rm and, } \ \ \forall w = w'w''\ \
 \ |w'|_a \geq |w'|_b  $$

A Dyck path is {\em primitive} if it is not the concatenation of two Dyck paths
or equivalently if it is equal to $aw'b$, where $w'$ is a Dyck path.
Such a path is usually drawn as a sequence of segments in the
cartesian plane
starting at the point $(0,0)$ and going from the point $(x,y)$ to
$(x+1, y+1)$ for each letter $a$ and from the point $(x,y)$ to
$(x+1, y-1)$ for each letter $b$. The conditions on the occurrences
of the letters translates in  the fact that the path ends on the $x$'s
axis  and never crosses it.
The  path $aaabaabbbbaabb$  is drawn in the figure below.

\begin{figure}[h]
\begin{center}
\includegraphics[scale = 0.6]{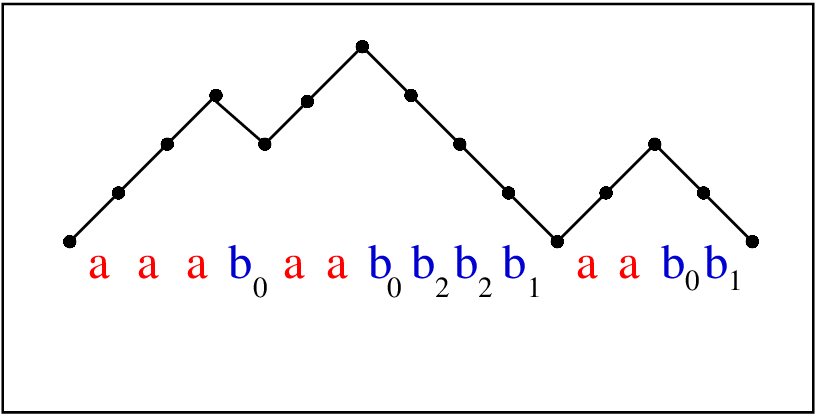}
\caption{A labeled Dyck path }
\end{center}
\end{figure}
\begin{Definition}\label{defLabel}

A {\em labeling} of a  Dyck path $w$ consists in assigning integer labels
to  the occurrences
of $b$, hence giving  a word $f$ on the alphabet
$\{a, b_0, b_1, \ldots , b_i, \ldots \}$ satisfying the following
conditions~:
\begin{itemize}
\item The occurrence $b_i$ in $f$ is  preceded by an $a$
if and only if $i = 0$, or equivalently $f$ has no
factor $ab_i$ for $i>0$, nor factor $b_jb_0$ for $j \geq 0$.
\item For each occurrence of $b_i$ in $f$ ($f = g b_ih$) with $i>0$, 
then $g$ does not end with an $a$
and  $i \leq |g|_a - |g|_b$ (where $|g|_b$ denotes $\sum _{i\geq 0}|g|_{b_i}$)
\end{itemize}
\end{Definition}

\subsection{From permutations to labeled Dyck paths}
We associate a labeled Dyck path $f$  with a permutation
$\alpha$ by the  algorithm described below.
\ The Dyck path associated with the permutation records
how the permutation is built using
two operations: the creation of  a new cycle and the insertion of an element
in a free position inside a cycle already created.

The elements $1, 2, \ldots, n$ are considered in that order.\
When $i$ is the smallest element of a  cycle of length $k$  in
$\alpha$, a new cycle is created,
$i$ is inserted as the first element of that
cycle,
and $k-1$ free positions are
created.
When $i$ is not the
smallest element of its cycle it has to be
inserted in a free position inside a cycle already
created and it is necessary to indicate  which free
position it is. For that, a {\em pivot}   element is needed, this pivot will be
the smallest element among all those which were inserted and which have, 
in their cycle, 
immediately at their  right, a free position. For instance after
considering element 1, then the pivot is 1 if it is not a fixed point
of $\alpha$ (if it is, there is no pivot).

The numbering of the free positions will begin at the pivot, the free one
at its right is numbered 1, the next free one to the right will be numbered 2, then
 proceed cyclically returning at the beginning of the
permutation after having numbered all the free positions at the right
of  the pivot.

We now describe  the bijection in detail, let 
$\alpha$ be a permutation in $\Sn$ then:

\begin{itemize}
\item With each  $i$ ($1\leq i \leq n$),
  is associated a word $f_i$ given by~:
\begin{enumerate}
\item  If $i$ is the smallest  element of a  cycle of length
  $k$ in $\alpha$, then $f_i = a^kb_0$,
\item else  $f_i = b_j$, where $j$ is the $j$-th free position (starting from
  the pivot) in which $i$ is inserted.

\end{enumerate}
\item The word $f$ associated with the permutation $\alpha$ is the
concatenation $f_1f_2 \ldots f_n$ of the $f_i$'s.
\end{itemize}

We will denote $\Delta(\alpha) = f$ the labeled Dyck path associated with
$\alpha$ by the above algorithm.

\subsubsection{Example}

Consider the permutation
$$\alpha \ \ = \ \ (1, 3, 5, 9) (2, 7, 6) (4, 8)\ \ = 3, 7, 5, 8, 9,
2, 6, 4, 1 .$$
The determination of the word $f= \Delta(\alpha)$
 is illustrated in
the figure below, where at each step of the algorithm the pivot is
represented by an arrow pointing to it.

\begin{figure}[h]
\begin{center}
\includegraphics{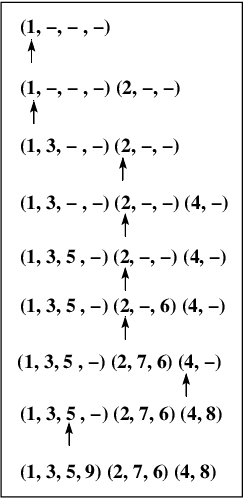}
\caption{Associating a labeled Dyck path to  a permutation }
\end{center}
\end{figure}

This gives:
$$ f_1 = aaaab_0, f_2 = aaab_0, f_3 = b_1,  f_4 = aab_0,
f_5 = b_4, f_6 = b_2, f_7 = b_1, f_8 = b_1, f_9 = b_1$$
Hence $ f \ \ = \ \ aaaab_0aaab_0b_1aab_0b_4b_2b_1b_1b_1.$

Note that the word $f$ is a Dyck path since each cycle
of the permutation is considered for the first time by its smallest element,
creating an excedence of $k-1$ for the number of occurrences of the $a$'s
with respect to those of the $b$'s.\ Then a letter $b$ is 
written for each element of
the cycle, so that there cannot be an excedence of the number of occurrences of $b$'s compared to that of $a$'s.

Note that this is a bijection since from a labeled Dyck path $f$
 one can obtain the  permutation $\alpha$  from $f$ by:

\bigskip

{\bf Inverse algorithm.} Let $f$ be a labeled Dyck path,

\begin{itemize}
\item Consider the occurrences of the letters $b$ from left to right in
  $f$
\item If  the $i$-th occurrence of $b$ in $f$   is $b_0$
let $k$ be the number of $a$ immediately
  before it, then open a new cycle of length $k$ with $i$ as first element and
$k-1$ free positions in it.
\item If the $i$-th occurrence of a $b$ is labeled $b_p$, with $p>0$, then
put $i$ in the $p$-th free position of the opened cycles starting from
the pivot.
\end{itemize}

\subsection{Characterization}
\begin{Theorem}\label{thBij1}
Let $\alpha$ be a permutation and $f= \Delta(\alpha)$, then:
\begin{itemize}

\item
The number of cycles in $\alpha$ is equal to the number of occurrences of
$b_0$ in $f$.

\item
The permutation $\alpha$ is indecomposable if and only if the word $f$
is primitive (i. e, it is not the concatenation of two Dyck paths).

\item
If the permutation $\alpha$ is indecomposable, then the number
of left-to-right maxima is equal to the number of occurrences of $b_1$
in $f$.

\end{itemize}

\end{Theorem}

\begin{proof}
\begin{itemize}
\item Since each smallest element in a cycle inserts a $b_0$ in $f$ the
number of cycles is equal to the number of occurrences of $b_0$

\item The  permutation is decomposable if and only if there is no free slot
at a step  before the very  end of the algorithm. But this exactly means that there is
a left factor of $f$  which is a Dyck path.

\item The smallest element $i$  of a cycle
of  length greater that 1 of a permutation $\alpha$
cannot be a left-to-right maximum of $\alpha$ since $i < a_i$.
If  a fixed point $j$  of $\alpha$ is a left-to-right maximum
then $j = a_j$ and $a_i <j$ for $i<j$ implies that $\{a_1, \ldots , a_{j-1}\}
= \{1,2 , \ldots , j-1\}$ showing that $\alpha$ is decomposable.
Moreover a $b_1$ corresponds to an element written immediately after the pivot,
at some
 step of the algorithm; this element  is the image of the pivot which is
the smallest element already considered with a free position after it; then
it gives  a left-to-right maximum.

\end{itemize}

\end{proof}

\noindent
{\bf Remark.} If $\alpha$ is decomposable, then a fixed point may
be one of its  left-to-right maximum, hence the number of left-to right maxima
for a permutation $\alpha$ with associated labeled Dyck path $f$ is
not less than the number of occurrences of $b_1$ in $f$ and not greater than
this number augmented by the number of fixed points of $\alpha$ (which is
also
the number of occurrences of $b_0$ preceded by exactly one $a$).
\medskip

\begin{coro}
The number of indecomposable permutations of $\Sn$ with $p$ cycles and $q$
left-to-right maxima is equal to the number of primitive labeled Dyck paths
of length $2n$ with $p$ occurrences of $b_0$ and $q$ occurrences of $b_1$.
\end{coro}

\subsection{Permutations by numbers of left-to-right maxima and
right-to-left minima}
In \cite{robletViennot}, E. Roblet and X. G. Viennot define a
bijection similar to $\Delta$, between permutations and another kind of labeled Dyck paths. Their labeling is on the alphabet $\{a, b_1, b_2, \ldots , b_i, \ldots\}$ and such that if $w'$ is a labeled Dyck word, in there sense,
 then~:
\begin{itemize}

\item Any occurrence in $w'$ of $b_i$ preceded by an $a$ is such that
$i=1$ 

\item Each occurrence of $b_i$ in $w'$ such that $w' = gb_i h$ 
satisfies $1 \leq i \leq |g|_a - |g|_b$.
\end{itemize}

Let us  denote by $\Delta'(\alpha)$ the labeled Dyck path obtained from
the permutation $\alpha$ by the bijection of E. Roblet and
X. Viennot.  The main feature of the bijection is  that the number
 of factors  $ab_1$ in
$w' = \Delta'(\alpha)$ is equal to the number of right-to-left minima of
$\alpha$, and the number of left-to-right maxima in $\alpha$ is equal
to the number of occurrences of $b_k\ \ (w = gb_kh)$ such that
$k = |g|_a -  |g|_b$. Moreover  like for $\Delta$, $\alpha$ is indecomposable if
and only if $w'$ is a primitive Dyck word. 
 Note that an occurrence of $b_1$ preceded
by an $a$ may correspond to both a left-to-right maximum and
right-to-left minima of $\alpha$ if $ |g|_a -  |g|_b = 1$, but this cannot
happen if $\alpha$ is indecomposable.

\begin{coro}
The number of permutations of $\Sn$ with $p$ left-to-right maxima and
$q$ cycles is equal to the number of  permutations of $\Sn$ with
$p$ left-to-right maxima and
$q$ right-to-left minima.
\end{coro}

 \begin{proof}
We first show that this is true for indecomposable permutations.
Let $\theta$ be an indecomposable permutation and let $w =
\Delta'(\theta)$, then replace each occurrence of $b_1$
in $w$  preceded by an $a$ by $b_0$
and each occurrence of $b_i  \ (w = gb_kh)$ not preceded by an $a$ by
$b_j$, where $j = |g|_a - |g|_b +1 -i$. Denote $w'$ the word such obtained,
it satisfies the conditions of Definition \ref{defLabel}, then
$\theta' = \Delta^{-1}(w')$ is well defined. $\theta'$  has as many
left-to-right maxima as $\theta$ and as many cycles as $\theta$ has
right-to-left minima, hence this 
gives the result for indecomposable permutations.

In order to complete  the proof it suffices to use the fact that any
permutation can be decomposed as 
 the concatenation of indecomposable permutations,
and to observe that in this decomposition the numbers of cycles,
left-to-right maxima and right-to-left minima of the permutation
are the  sum of the corresponding parameters of the indecomposable components.
 \end{proof}
\medskip

\noindent
{\bf Remark.} Note that this corollary gives another proof of the
symmetry of the statistics for the number of cycles, and number of
left-to-right maxima. Indeed,  this symmetry is clear for
the parameters left-to-right maxima and right-to-left minima since
they are exchanged by the inverse operation on permutations.

\section{Bivariate polynomials associated to Dyck paths}
In this section we introduce two polynomials whose construction is
very similar to those introduced by P. Deleham (see \cite{deleham})
who defined the $\Delta$-operator. This operator consists 
in assigning a polynomial to a Dyck word as the product of binomials
associated to each occurrence of a letter $b$, these binomials depend 
only of the height of the occurrence.

\subsection{Polynomial associated to a given Dyck path}
A factor $ab$ in a Dyck path is often called a {\em peak}.
We go back to the labeled Dyck paths considered in Section 3, and
associate a polynomial $L(w)$ in two variables $x, y$ with a Dyck
path $w$ as follows:
\begin{Definition}
 Let $w$ be a Dyck path of length $2n$ and
let $w'_1b, w'_2b, \ldots w'_nb$ be the left factors of $w$ ending with
 an occurrence of $b$, that is~ $w = w'_ibw''_i$.

The polynomial
$L(w)$ is the product of binomials $u_i$
associated with each  $w'_i$ by the following rule:

\begin{itemize}
\item  If $w'_i$ ends with $a$, then $u_i = x$,
\item else  $u_i= y +h_i$, where $h_i$ is given by~
$ h = |w'_ib|_a -|w'_ib|_b$.
\end{itemize}
\end{Definition}
Remark that the number of possible values of $i$ in  $b_i$ for
 a labeling of $w$
is $h_i+1$, hence $L(1,1)$ is exactly the number of possible labelings of $w$.

An example of a polynomial associated with a Dyck path is given below.
\begin{figure}[h]\label{figDyck}
\begin{center}
\includegraphics{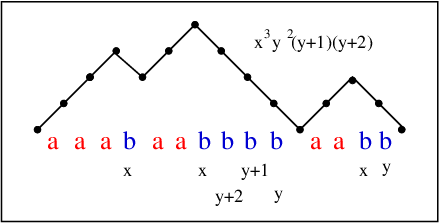}
\caption{The polynomial associated with the Dyck path $aabaabbbbaabb$. }
\end{center}
\end{figure}

\begin{propo}\label{propDyck1}
The polynomial $L(w) = x^p \sum_{i=1}^na_i y^i$ associated with a Dyck path
$w$ is such that $p$ is the number of factors $ab$ in $w$, and $a_i$ is
the number of labelings of $w$ such that $i$ occurrences of $b$ are
labeled $b_1$.
\end{propo}
\begin{proof}
The first part is immediate since each  $u_i= x$ corresponds to  a
factor $ab$ in $w$. For the second part, note that the possible
labelings $b_i$ of an occurrence of $b$ not preceded by an $a$
are such that $1 \leq i \leq |w'b|_a - |w'|_b= h_i+1$. Since we have
in this case $u_i = y + h_i$ then $y$ in $u_i$ may be interpreted as
the labeling $b_1$ for that occurrence, the other labelings
corresponding to the integer $h_i$.
\end{proof}

\begin{propo}\label{propDyck2}

Let  $w= a^{k_1}b^{\ell_1}a^{k_2}b^{\ell_2} \ldots a^{k_p}b^{\ell_p} $
be  a primitive Dyck path, where $k_i >0, \ell_i >0$
then the coefficient $a_i$  in
the  polynomial $L(w)= x^p \sum a_iy^i $
 is the number of indecomposable permutations
$\theta$ such that~:

\begin{itemize}
\item $\theta$ has $p$ cycles which smallest elements are
$1, 1+ \ell_1, 1+ \ell_1 + \ell_2, \ldots ,
1+ \ell_1+\ell_2 + \ldots + \ell_{p-1} $;
\item these cycles are of respective lengths $k_1, k_2, \ldots k_p$;
 \item $\theta$ has $i$ left-to-right maxima.
\end{itemize}

\end{propo}

\begin{proof}
Follows from Proposition \ref{propDyck1} and Theorem \ref{thBij1}
\end{proof}

\subsection{Sum of the polynomials for all  paths of a given length}
Let $D_n$ be the set of Dyck paths of length $2n$ and $D'_n$ the set
of primitive Dyck paths of length $2n$; clearly $D'_n = aD_{n-1}b$.\
We consider the polynomials $L_n(x,y) = \sum_{w \in D_n} L(w)$
and $L'_n(x,y) = \sum_{w \in D'_n} L(w)$.\
An example of the polynomials associated with the five Dyck paths of
length 6 and allowing to compute $L_3 = x^3 + 3x^2y+ xy^2 +xy$ is given below:

\begin{figure}[h]
\begin{center}
\includegraphics{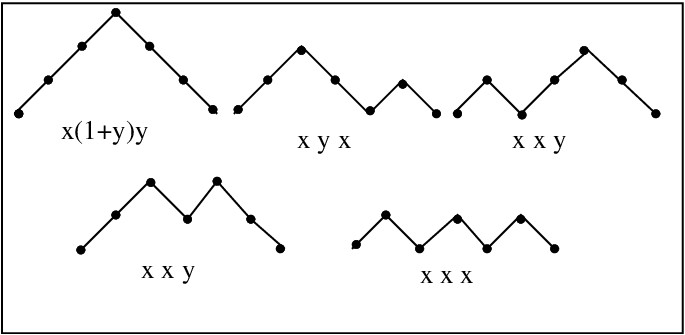}
\caption{A polynomial enumerating the number of cycles and
  left-to-right maxima}
\end{center}
\end{figure}
\begin{propo}
The polynomials $L_n$ and $L'_n$ satisfy the following relations
for $n>1$~:

$$ L'_n (x, y) = y L_{n-1}(x, y+1)$$
$$ L_n (x, y) = L'_n(x, y) + \sum_{p=1}^{n-1} L'_p(x,y) L_{n-p}(x,y)$$

\end{propo}

\begin{proof}
For the first relation note that
  the value of $L(awb)$ is obtained from $L(w)$ by replacing  $y+1$ for $y$
and multiplying by $y$.
To prove the second relation observe that a Dyck path of length $2n$
is either primitive or
the concatenation of a primitive Dyck path of length $2p$ where $1 \leq p < n$
and another one
(not necessarily primitive) of length $2n -2p$.
\end{proof}

The first few values of these polynomials are:
$$ L_1 = x,\ \  L'_1 = x,\ \ \ \ \ \  L'_2 = xy ,\ \ \  L_2 = xy + x^2$$
$$ L'_3 = xy^2+ x^2y + xy, \ \ \ \ \ L_3 = x^3 + xy^2+ 3 x^2y + xy$$
$$L'_4 = xy^3 + 3 x y^2 + 3y^2x^2 + 2xy +3yx^2 +yx^3$$
$$L_4 = xy^3 + 3 x y^2 + 6y^2x^2 + 2xy +5yx^2 +6yx^3 + x^4$$

\begin{coro}
The coefficient of $x^py^q$ in the polynomial $L'_n$ is the number of 
indecomposable permutations of $\Sn$ with $p$ cycles and $q$ left-to-right
maxima. Moreover, the 
polynomials $L'_n$
 are symmetric in $x,y$: for all $n>1$ we have
$$L'_n(x,y) \ \ = \ \ L'_n(y,x)$$
\end{coro}

\begin{proof}
The first part is a direct consequence of Theorem 3 and  Proposition 5.
The symmetry of the polynomials  follows from the fact that the number of
indecomposable  permutations of $\Sn$ having $p$ cycles and $q$
left-to-right maxima is equal to the number of
indecomposable  permutations of $\Sn$ having $q$ cycles and $p$
left-to-right maxima.
\end{proof}

\begin{coro}
The number of permutations of $\Sn$ having $p$ cycles and $q$
left-to-right maxima is the coefficient of $z^nx^py^q$ in the power
series~:
$$\frac{1}{1- zL'_1-z^2L'_2 -z^3 L'_3 \ldots}$$
\end{coro}

\section{Indecomposable fixed points free involutions and maps}

A {\em map}  is a pair of permutations $\sigma, \alpha$ where $\alpha$,
is a  fixed point free involution. Maps may be considered as a subset
 of the set hypermaps, but conversely a hypermap may be considered as a
 bipartite map \cite{walsh}.
Maps are  an important combinatorial and algebraic tool for dealing
with embeddings of graphs in surfaces (\cite{tutte2} \cite{edmonds})
they are sometimes  called {\em rotation system} (\cite{moharThomassen}).

\subsection{Fixed point free involutions}
A fixed point free involution is permutation where all cycles have
length $2$, the number of fixed point free involutions of ${\mathcal
  S}_{2m}$ is the double factorial~:
 $$(2m-1)!! \ = \ (2m-1) (2m-3) \ldots
  (3) (1) \ \ = \frac{(2m)!}{m! 2^m}$$
As expected, an indecomposable fixed point free involution is  a
 fixed point free involution which is
indecomposable as a permutation.\ The number  $i_m$ of 
indecomposable fixed point free involutions of ${\mathcal S}_{2m}$
satisfies the following inductive relation, which is very similar to that satisfied by the
 indecomposable permutations.

\begin{propo}
$$i_m  \ \ = \ \ (2m-1)!! - \sum_{p=1}^{m-1} i_p (2m -2p-1)!!$$
\end{propo}

\begin{proof}
A decomposable fixed point free involution  of ${\mathcal S}_{2m}$ can
be  written as the concatenation of an indecomposable fixed point free
involution  of ${\mathcal S}_{2p}$ ($1<p<m$) 
 and a  fixed point free involution
of  ${\mathcal S}_{2m-2p}$.
\end{proof}

\subsection{Bijection}

The bijection between rooted hypermaps and indecomposable permutations
specializes for maps and fixed point free involutions in the following
way.

\begin{propo}
There exist a bijection $\Psi'$ between the set of  indecomposable
fixed point free involutions on  $1,2, \ldots,  2m+2$
and the set of rooted maps on
 $1,2, \ldots , 2m$.
Moreover for   $(\sigma, \alpha) = \Psi'(\theta)$,  the number of
cycles of $\sigma$  is equal to the number of
left-to-right maxima of $\theta$.
\end{propo}

\begin{proof}
The bijection $\Psi$ associates with an indecomposable fixed point free
involution $\theta$ of ${\mathcal S}_{2m+2}$
a hypermap $(\sigma', \alpha')$ such that
$\alpha'$ has $m$ cycles of length 2 and 1 cycle of length 1.
The element in this cycle is $j= \theta(2m+2)$, clearly this $j$ is
not a left-to-right maxima of $\theta$. Consider the   pair
of permutations  $(\sigma, \alpha)$ obtained by deleting $j$ from its
cycle in $\sigma'$ and deleting the cycle of length 1 in $\alpha$,
then renumbering the darts by  $\phi(i) = i -1$ for $i > j$.  Then
$(\sigma, \alpha)$  is a map.
\end{proof}

Note that this construction was described in detail
  by P. Ossona de Mendez and P.Rosenstiehl (see \cite{ossonaRosenstiehl2}).
\begin{coro}
The number of rooted maps with $m$ edges is equal to $i_{m+1}$, and the
number of maps with $n$ vertices and $m$ edges is equal to the number
number of indecomposable fixed point free involutions with $n$
left-to-right maxima.
\end{coro}

\subsection{Involutions and Dyck paths}

Theorem \ref{thBij1} specializes for fixed-point free involutions in
the following.

\begin{propo}
Let $\theta$ be a fixed point free involution of  ${\mathcal S}_{2m}$
and $w$ the labeled Dyck path
associated with it then:
\begin{itemize}
\item
The length of $w$ is $4m$ and there  are $m$ factors of $aab_0$ in  $w$.
\item
The fixed point free involution
 $\theta$ is indecomposable if and only if the word $w$
is primitive.
\item
If the permutation $\theta$ is indecomposable, then the number
of left-to-right maxima is equal to the number of occurrences of $b_1$
in $w$.
\end{itemize}
\end{propo}

Replacing in $w$ each factor $aab_0$ by $a$ gives a new kind of
labeled Dyck path $w'$ in which there is no occurrence of $b_0$ and
in which  for any occurrence $w'= ub_iv$ of $b$ we have
$$ |u|_a - |u|_b \geq i$$

This suggests to associate the  polynomial $M(w)$ with a Dyck path
$w$ by:
\begin{Definition}
 Let $w$ be a Dyck path of length $2n$ and
let $w'_1b, w'_2b, \ldots w'_nb$ be the left factors of $w$ ending with
 an occurrence of $b$, that is~ $w = w'_ibw''_i$. The polynomial
$M(w)$ is the product of binomials $v_i$
associated with each occurrence of $b$ by the following rule:
 Let $h_i = |w'_ib|_a -|w'_ib|_b$
 $v_i = y + h_i$
\end{Definition}

\subsection{An enumeration formula}
We define the polynomials $M_m(y)$, and  $M'_m(y)$ by
$M_m(y) = \sum_{w \in D_m} M(w)$
and $M'_m(y) = \sum_{w \in D'_m} M(w)$.

\begin{propo}\label{propEqV}
The polynomials $M_m(y)$ satisfy the following equations:
$$M_1(y) = M'_1(y) = y $$
and  for $m >1$:

$$ M'_m (y) = yM_{m-1}(1+y)$$
$$ M_m (y) = M'_m(y) + \sum_{p=1}^{n-1} M'_p(y) M_{m-p}(y)$$
Moreover, in $M'_m(y)$ the coefficient of $y^n$ is the number
of rooted  maps with $m-1$ edges and $n$ vertices.
\end{propo}
The first values of $M'_m$ are:
$$ M'_2 = y^2 + y, \ \ M'_3 = 2y^3 + 5y^2 + 3y, \ \
M'_4 = 5 y^4 + 22y^3 + 32 y^2 + 15y$$

This gives a simpler proof of a formula given by
D. Arques and J. F. B\'eraud \cite{arquesBeraud}
(see also  \cite{walshLehman} for the computation
of the first values).

\begin{coro}
Let $U(z,y)$ be the formal power series enumerating rooted maps
by the number of edges and vertices,
$$U(z, y) \ \ \ = \sum_{m,p>0}\mu_{m,p}z^my^p   =\ \ \ \sum_{m\geq 1} z^mM'_m(y),$$
where $\mu_{m,p}$ is the number of rooted maps with $m$ edges and $p$ vertices.\
Then
$$U(z,y) \ \ = \ \ y + zU(z,y)U(z,y+1)$$

\end{coro}
\begin{proof}

By proposition \ref{propEqV}
we have~:
$$ \frac{M'_{m+1}(y)}{y} \ \ = \ \  M_m (y+1) = M'_m(y+1) + \sum_{p=1}^{n-1}
M'_p(y+1) M_{m-p}(y+1) \  $$
$$ \frac{M'_{m+1}(y)}{y} \ \ = M'_m(y+1) + \sum_{p=1}^{n-1}
M'_p(y+1) \frac{M'_{m+1-p}(y)}{y}$$
Hence~:
$$ M'_{m+1}(y) \ \ = y M'_m(y+1) + \sum_{p=1}^{n-1}
M'_p(y+1) M'_{m+1-p}(y)$$
Multiplying these  equalities by $z^{m+1}$ for all $m>1$ and adding
gives the result.
\end{proof}

Similar results were obtained recently by D. Drake for 
Hermite polynomials related to weighted involutions which he calls
matchings (see \cite{drake}).

\section*{Acknowledgments} The author wishes to thanks warmly,
Domnique Foata, Antonio Mach\'i, Gilles Schaeffer and Xavier Viennot
for fruitful discussions and comments on earlier versions of this paper.